\documentclass[leqno,11pt]{amsart}
\usepackage{amsmath}
\usepackage{amssymb,amscd}
\usepackage{amsthm}
\usepackage{latexsym}
\usepackage[myheadings]{fullpage}
\usepackage{color}  

\newtheorem{theorem}{Theorem}[section]
\newtheorem{lemma}[theorem]{Lemma}
\newtheorem{definition}[theorem]{Definition}
\newtheorem{proposition}[theorem]{Proposition}
\newtheorem{corollary}[theorem]{Corollary}
\newtheorem{remark}[theorem]{Remark}
\newtheorem{example}[theorem]{Example}
\newtheorem{examples}[theorem]{Examples}

\newcommand{\m}{\mathfrak m}

\newcommand{\du}{S \! \Join^b \! I}

\def\opn#1#2{\def#1{\operatorname{#2}}} 
\opn\spec{Spec}
\opn\depth{depth}
\opn\height{ht}
\opn\chara{char}
\opn\gr{gr}
\opn\ord{ord}
\opn\Ap{Ap}
\opn\F{F}
\opn\PF{PF}
\opn\L{L}
\opn\Max{Max}

\title{Almost canonical ideals and GAS numerical semigroups}

\author{Marco D'Anna}
\address{Marco D'Anna - Dipartimento di Matematica e Informatica - Universit\`a degli Studi di Catania - Viale Andrea Doria 6 - 95125 Catania - Italy}\email{mdanna@dmi.unict.it}

\author{Francesco Strazzanti}
\address{Francesco Strazzanti - Dipartimento di Matematica e Informatica - Universit\`a degli Studi di Catania - Viale Andrea Doria 6 - 95125 Catania - Italy}
\email{francesco.strazzanti@gmail.com}

\thanks{The authors were supported by the project ``Propriet\`a algebriche locali e globali di anelli associati a curve e  
ipersuperfici'' PTR 2016-18 - Dipartimento di Matematica e Informatica - Universit\`a degli Studi di Catania \\
The second author was also supported by INdAM, more precisely he was ``titolare di un Assegno di Ricerca dell'Istituto Nazionale di Alta Matematica''.}


\date{\today}

\begin{document}

\begin{abstract}
We propose the notion of GAS numerical semigroup which generalizes both almost symmetric and 2-AGL numerical semigroups. 
Moreover, we introduce the concept of almost canonical ideal which generalizes the notion of canonical ideal in the same way almost symmetric numerical semigroups generalize symmetric ones. We prove that a numerical semigroup with maximal ideal $M$ and multiplicity $e$ is GAS if and only if $M-e$ is an almost canonical ideal of $M-M$. This generalizes a result of Barucci about almost symmetric semigroups and a theorem of Chau, Goto, Kumashiro, and Matsuoka about 2-AGL semigroups. We also study the transfer of the GAS property from a numerical semigroup to its gluing, numerical duplication and dilatation.
\end{abstract}


\maketitle

\section*{Introduction}

The notion of Gorenstein ring turned out to have great importance in commutative algebra, algebraic geometry and other mathematics areas and in the last decades many researchers have developed generalizations of this concept obtaining rings with similar properties in certain respects. With this aim, in 1997  Barucci and Fr\"oberg  \cite{BF} introduced the notion of almost Gorenstein ring, inspired by numerical semigroup theory. We recall that a numerical semigroup $S$ is simply an additive submonoid of the set of the natural numbers $\mathbb{N}$ with finite complement in $\mathbb{N}$. The simplest way to relate it to ring theory is by associating with $S$ the ring $k[[S]]=k[[t^s \mid s \in S]]$, where $k$ is a field and $t$ is an indeterminate. Actually it is possible to associate a numerical semigroup $v(R)$ with every one-dimensional analytically irreducible ring $R$. In this case a celebrated result of Kunz \cite{K} ensures that $R$ is Gorenstein if and only if $v(R)$ is a symmetric semigroup, see also \cite[Theorem 4.4.8]{BH} for a proof in the particular case of $k[[S]]$. 
In \cite{BF} the notions of almost symmetric numerical semigroup and almost Gorenstein ring are introduced, where the latter is limited to analytically unramified rings. It turns out that $k[[S]]$ is almost Gorenstein if and only if $S$ is almost symmetric.

More recently this notion has been generalized in the case of one-dimensional local ring \cite{GMP} and in higher dimension \cite{GTT}. Moreover, in \cite{CGKM} it is introduced the notion of $n$-AGL ring in order to stratify the Cohen-Macaulay rings. Indeed a ring is almost Gorenstein if and only if it is either $1$-AGL or $0$-AGL, with $0$-AGL equivalent to be Gorenstein. In this respect $2$-AGL rings are near to be almost Gorenstein and for this reason their properties have been deepened in \cite{CGKM,GIT}. In \cite{CGKM} it is also studied the numerical semigroup case, where $2$-AGL numerical semigroups are close to be almost symmetric.

In this paper we introduce the class of {\em Generalized Almost Symmetric numerical semigroups}, briefly GAS numerical semigroups, that includes symmetric, almost symmetric and 2-AGL numerical semigroups, but not 3-AGL. Moreover, if $S$ has maximal embedding dimension and it is GAS, then it is either almost symmetric or 2-AGL.
Our original motivation to introduce this class is a result on 2-AGL numerical semigroups that partially generalize a property of almost symmetric semigroups. More precisely, let $S$ be a numerical semigroup with multiplicity $e$ and let $M$ be its maximal ideal. In \cite[Corollary 8]{BF} it is proved that $M-M$ is symmetric if and only if $S$ is almost symmetric with maximal embedding dimension. If we do not assume that $S$ has maximal embedding dimension, it holds that $S$ is almost symmetric if and only if $M-e$ is a canonical ideal of $M-M$ (indeed $S$ has maximal embedding dimension exactly when $M-e=M-M$, see \cite[Theorem 5.2]{B}). In \cite[Corollary 5.4]{CGKM} it is shown that $S$ is 2-AGL if and only if $M-M$ is almost symmetric and not symmetric, provided that $S$ has maximal embedding dimension. 

Hence, it is natural to investigate what happens to $M-M$, for a 2-AGL semigroup, if we do not make any assumptions on its embedding dimension. It turns out that $M-e$ is an ideal of $M-M$ that satisfies some equivalent conditions, that are the analogue for ideals to the  defining conditions of almost symmetric semigroup (cf. Definition 2.1 and Proposition \ref{almost canonical ideal});
for this reason we called the ideals in this class \emph{almost canonical ideals}. However the converse is not true:
there exist numerical semigroups $S$ such that $M-e$ is an almost canonical ideal of $M-M$, but that are not 2-AGL.
This fact lead us to look for those numerical semigroup satisfying this property, and we found that these 
semigroups naturally generalize 2-AGL semigroups (this is evident if we look at $2K\setminus K$, where $K$ is the canonical ideal of $S$, cf. Proposition 3.1 and Definition 3.2); moreover, as we said above this class coincides with the union of 2-AGL and almost symmetric semigroups, if we assume maximal embedding dimension; hence we called them Generalized Almost Symmetric (briefly GAS). It turns out that GAS semigroups are interesting under many aspects; for example, if $S$ is GAS, it is possible to control both the semigroup generated by its canonical ideal (that plays a fundamental role in \cite{CGKM}; cf. Theorem \ref{Livelli più alti}) and its pseudo-Frobenius numbers (cf. Proposition \ref{PF GAS}). 

Hence, in this paper, after recalling the basic definitions and notations,  we introduce, in Section 2, the concept of almost canonical ideal. 
We show under which respect they are a generalization of canonical ideals and we notice that, similarly to the canonical case, a numerical semigroup $S$ is almost symmetric if and only if it is an almost canonical ideal of itself. Moreover, we prove several equivalent conditions for a semigroup ideal to be almost canonical (cf. Proposition \ref{almost canonical ideal}) and we show how to find all the almost canonical ideals of a numerical semigroup and to count them (Corollary \ref{Number of almost canonical ideals}).

In Section 3 we develop the theory of GAS semigroups proving many equivalent conditions (see Proposition \ref{Characterizations GAS}), exploring their properties (cf. Theorem \ref{Livelli più alti} and Proposition \ref{PF GAS}) and relating them with other classes of numerical semigroups that have been recently introduced to generalize almost symmetric semigroups. The main result is Theorem \ref{T. Almost Canonical ideal of M-M}, where it is proved that $S$ is GAS if and only if $M-e$ is an almost canonical ideal of $M-M$. 

Finally in Section 4 we study the transfer of the GAS property from $S$ to some numerical semigroup constructions: gluing in Theorem \ref{gluing}, numerical duplication in Theorem \ref{Numerical duplication S-<K>} and dilatation in Proposition \ref{dilatation}.

Several computations are performed by using the GAP system \cite{GAP} and, in particular, the NumericalSgps package \cite{DGM}.

\section{Notation and basic definitions}

A numerical semigroup $S$ is a submonoid of the natural numbers $\mathbb{N}$ such that $|\mathbb{N} \setminus S| < \infty$. Therefore, there exists the maximum of $\mathbb{N} \setminus S$ that is said to be the Frobenius number of $S$ and it is denoted by $\F(S)$. Given $s_1, \dots, s_{\nu} \in \mathbb{N}$ we set $\langle s_1, \dots, s_{\nu} \rangle=\{\lambda_1 s_1 + \dots + \lambda_{\nu} s_{\nu} \mid \lambda_1, \dots, \lambda_{\nu} \in \mathbb{N} \}$ which is a numerical semigroup if and only if $\gcd(s_1, \dots, s_{\nu})=1$. We say that $s_1, \dots, s_{\nu}$ are minimal generators of $\langle s_1, \dots, s_{\nu} \rangle$ if it is not possible to delete one of them obtaining the same semigroup. It is well-known that a numerical semigroup have a unique system of minimal generators, which is finite, and its cardinality is called embedding dimension of $S$. The minimum non-zero element of $S$ is said to be the multiplicity of $S$ and we denote it by $e$. It is always greater than or equal to the embedding dimension of $S$ and we say that $S$ has maximal embedding dimension if they are equal. 
Unless otherwise specified, we assume that $S \neq \mathbb{N}$. 

A set $I \subseteq \mathbb{Z}$ is said to be a relative ideal of $S$ if $I+S\subseteq I$ and there exists $z \in S$ such that $z+I \subseteq S$. If it is possible to chose $z=0$, i.e. $I \subseteq S$, we simply say that $I$ is an ideal of $S$. Two very important relative ideals are $M(S)=S\setminus \{0\}$, which is an ideal and it is called the maximal ideal of $S$, and $K(S)=\{x \in \mathbb{N} \mid \F(S)-x \notin S\}$. We refer to the latter as the standard canonical ideal of $S$ and we say that a relative ideal $I$ of $S$ is canonical if $I=x+K(S)$ for some $x \in \mathbb{Z}$. If the semigroup is clear from the context, we write $M$ and $K$ in place of $M(S)$ and $K(S)$.
Given two relative ideals $I$ and $J$ of $S$, we set $I-J = \{x \in \mathbb{Z} \mid x+J \subseteq I\}$ which is a relative ideal of $S$. For every relative ideal $I$ it holds that $K-(K-I)=I$, in particular $K-(K-S)=S$. Moreover, an element $x$ is in $I$ if and only if $\F(S)-x \notin K-I$, see \cite[Hilfssatz 5]{J}. As a consequence we get that the cardinalities of $I$ and $K-I$ are equal.  
Also, if $I \subseteq J$ are two relative ideals, then $|J \setminus I|=|(K-I)\setminus (K-J)|$.
We now collect some important definitions that we are going to generalize in the next section.

\begin{definition} \label{Basic definitions} \rm Let $S$ be a numerical semigroup.
\begin{enumerate}
\item The {\it pseudo-Frobenius numbers} of $S$ are the elements of the set $\PF(S)=(S-M)\setminus S$.
\item The {\it type} of $S$ is $t(S)=|\PF(S)|$.
\item $S$ is {\it symmetric} if and only if $S=K$.
\item $S$ is {\it almost symmetric} if and only if $S-M=K \cup \{\F(S)\}$.
\end{enumerate}
\end{definition}

We note that $M-M=S \cup \PF(S)$. Given $0 \leq i \leq e-1$, let $\omega_i$ be the smallest element of $S$ that is congruent to $i$ modulo $e$. 
A fundamental tool in numerical semigroup theory is the so-called Ap\'ery set of $S$ that is defined as $\Ap(S)=\{\omega_0=0, \omega_1, \dots, \omega_{e-1}\}$.  
In $\Ap(S)$ we define the partial ordering $x \leq_S y$ if and only if $y= x+s$ for some $s \in S$ and we denote the maximal elements of $\Ap(S)$ with respect to $\leq_S$ by ${\rm Max}_{\leq_S}(\Ap(S))$. With this notation $\PF(S)=\{\omega -e \mid \omega \in {\rm Max}_{\leq S}(\Ap(S)) \}$, see \cite[Proposition 2.20]{RG}. 
We also recall that $S$ is symmetric if and only if $t(S)=1$, that is also equivalent to say that $k[[S]]$ has type $1$ for every field $k$, i.e. $k[[S]]$ is Gorenstein. Also for almost symmetric semigroups many useful characterizations are known, for instance it is easy to see that our definition is equivalent to $M+K \subseteq M$, but see also \cite[Theorem 2.4]{N}
for another useful characterization related to the Ap\'ery set of $S$ and its pseudo-Frobenius numbers.

\section{Almost canonical ideals of a numerical semigroup}

If $I$ is a relative ideal of $S$, the set $\mathbb{Z}\setminus I$ has a maximum that we denote by $\F(I)$. We set $\widetilde{I}=I+(\F(S)-\F(I))$, that is the unique relative ideal $J$ isomorphic to $I$ for which $\F(S)=\F(J)$, and we note that $\widetilde{I} \subseteq K \subseteq \mathbb{N}$ for every $I$. 
The following is a generalization of Definition \ref{Basic definitions}.

\begin{definition} \rm 
Let $I$ be a relative ideal of a numerical semigroup $S$.
\begin{enumerate}
\item The {\it pseudo-Frobenius numbers} of $I$ are the elements of the set $\PF(I)=(I-M)\setminus I$.
\item The {\it type} of $I$ is $t(I)=|\PF(I)|$.
\item $I$ is {\it canonical} if and only if $\widetilde{I}=K$.
\item $I$ is {\it almost canonical} if and only if $\widetilde{I}-M=K \cup \{\F(S)\}$.
\end{enumerate}
\end{definition}

\begin{remark} \rm \label{Rem as}
{\bf 1.} $S$ is an almost canonical ideal of itself if and only if it is an almost symmetric semigroup. \\
{\bf 2.} $M$ is an almost canonical ideal of $S$ if and only if $S$ is an almost symmetric semigroup. Indeed, $M-M=S-M$, since $S \neq \mathbb{N}$. Moreover, $t(M)=t(S)+1$. \\ 
{\bf 3.} It holds that $K-M=K \cup \{\F(S)\}$. One containment is trivial, so let $x \in ((K-M) \setminus (K \cup \{\F(S)\}))$. Then $0 \neq \F(S)-x \in S$ and, thus, $\F(S)=(\F(S)-x)+x \in M+ (K-M) \subseteq K$ yields a contradiction. In particular, a canonical ideal is almost canonical. \\
{\bf 4.} Since $\F(S)=\F(\widetilde{I})$, it is always in $\widetilde{I}-M$. Moreover, we claim that $(\widetilde{I}-M) \subseteq K \cup \{\F(S)\}$. Indeed, if $x \in (\widetilde{I}-M)\setminus\{\F(S)\}$ and $x \notin K$, then $\F(S)-x \in M$ and $\F(\widetilde{I})=\F(S)=(\F(S)-x)+x \in \widetilde{I}$. 
In addition, $\widetilde{I}$ is always contained in $\widetilde{I}-M$ because it is a relative ideal of $S$. Hence, $I$ is an almost canonical ideal of $S$ if and only if 
$K \setminus \widetilde{I} \subseteq (\widetilde{I}-M)$.
\end{remark}

Given a relative ideal $I$ of $S$, the Ap\'ery set of $I$ is $\Ap(I)=\{i \in I \mid i-e \notin I\}$. As in the semigroup case, in $\Ap(I)$ we define the partial ordering $x \leq_S y$ if and only if $y= x+s$ for some $s \in S$ and we denote by ${\rm Max}_{\leq_S}(\Ap(I))$ the maximal elements of $\Ap(I)$ with respect to $\leq_S$.

\begin{proposition} Let $I$ be a relative ideal of $S$. The following statements hold:  
\begin{enumerate}
\item $\PF(I)= \{i-e \mid i \in {\rm Max}_{\leq_S}(\Ap(I)) \}$;
\item $I$ is canonical if and only if its type is $1$.
\end{enumerate}
\end{proposition}

\begin{proof}
(1) An integer $i \in I$ is in ${\rm Max}_{\leq_S}(\Ap(I))$ if and only if $i-e \notin I$ and $s+i \notin \Ap(I)$, i.e. $s+i-e \in I$, for every $s \in M$. This is equivalent to say that $i-e \in (I-M)\setminus I=\PF(I)$. \\
(2) Since $\F(S)\in \widetilde I-M$, we have $t(\widetilde I)=t(I)=1$ if and only if $\widetilde{I}-M=\widetilde I \cup \{\F(S)\}$. Therefore, a canonical ideal has type 1 by Remark \ref{Rem as}.3. Conversely, assume that $t(\widetilde{I})=1$ and let $x \notin \widetilde{I}$. Since $\widetilde{I} \subseteq K$, we only need to prove that $x \notin K$. By (1), there is a unique maximal element in $\Ap(\widetilde{I})$ with respect to $\leq_S$ and, clearly, it is $\F(S)+e$. Let $0 \neq \lambda \in \mathbb{N}$ be such that $x+ \lambda e \in \Ap(\widetilde{I})$. Then, there exists $y \in S$ such that $x+\lambda e + y = \F(S)+e$ and $x=\F(S)-(y+(\lambda-1)e) \notin K$, since $y+(\lambda-1)e \in S$.
\end{proof}

Let $g(S)=|\mathbb{N}\setminus S|$ denote the genus of $S$ and let $g(I)=|\mathbb{N}\setminus \widetilde{I}|$. We recall that $2g(S) \geq \F(S) + t(S)$ and the equality holds if and only if $S$ is almost symmetric, see, e.g., \cite[Proposition 2.2 and Proposition-Definition 2.3]{N}.

\begin{proposition} \label{almost canonical ideal}
Let $I$ be a relative ideal of $S$. Then
$g(I)+g(S) \geq \F(S) + t(I)$.
Moreover, the following conditions are equivalent:
\begin{enumerate}
\item $I$ is almost canonical; 
\item $g(I)+g(S)=\F(S)+t(I)$;
\item $\widetilde{I}-M=K-M$;
\item $K-(M-M) \subseteq \widetilde{I}$;
\item If $x \in \PF(I)\setminus \{\F(I)\}$, then $\F(I)-x \in \PF(S)$.
\end{enumerate}
\end{proposition}

\begin{proof}
Clearly, $t(I)=t(\widetilde{I})$ and $g(I)-t(\widetilde{I})=|\mathbb{N} \setminus \widetilde{I}|-|(\widetilde{I}-M)\setminus \widetilde{I}|=|\mathbb{N}\setminus (\widetilde{I}-M)|$. Moreover, since $\F(S)+1-g(S)$ is the number of the elements of $S$ smaller than $\F(S)+1$, it holds that $\F(S)-g(S)=|\mathbb{N}\setminus K|-1=|\mathbb{N} \setminus (K \cup {\F(S)})|$. We have $\widetilde{I}-M \subseteq K \cup \{\F(S)\}$ by Remark \ref{Rem as}.4, then $g(I)-t(I) \geq \F(S) -g(S)$ and the equality holds if and only if $\widetilde{I}-M = K \cup \{\F(S)\}$, i.e. $I$ is almost canonical. Hence, (1) $\Leftrightarrow$ (2). \\
(1) $\Leftrightarrow$ (3). We have already proved that $K-M=K \cup \{\F(S)\}$ in Remark \ref{Rem as}.3. \\
(1) $\Rightarrow$ (4). The thesis is equivalent to $M-M \supseteq K-\widetilde{I}$. Let $x \in K-\widetilde{I}$ and assume by contradiction that there exists $m \in M$ such that $x+m \notin M$. Then, $\F(S)-x-m \in K \cup \{\F(S)\}=\widetilde{I}-M$ and, so, $\F(S)-x \in \widetilde{I}$. Since $x \in K-\widetilde{I}$, this implies $\F(S) \in K$, that is a contradiction. \\
(4) $\Rightarrow$ (1). Let $x \in K$. It is enough to prove that $x \in \widetilde{I}-M$. Suppose by contradiction that there exists $m \in M$ such that $x+m \notin \widetilde{I}\supseteq K-(M-M)$. In particular, $x+m\notin K-(M-M)$ and so $\F(S)-(x+m) \in M-M$. This implies $\F(S)-x \in M$, that is a contradiction because $x \in K$. \\
(1) $\Rightarrow$ (5) We notice that $\PF(\widetilde{I})=\{x+\F(S)-\F(I) \mid x \in \PF(I)\}$. Let $x \in \PF(I)\setminus \{\F(I)\}$ and let $y=x+\F(S)-\F(I) \in \PF(\widetilde I) \setminus \{ \F(S)\}$. We first note that $\F(S)-y \notin S$, otherwise $\F(S)=y+(\F(S)-y) \in \widetilde I$. Assume by contradiction that $\F(S)-y \notin \PF(S)$, i.e. there exists $m \in M$ such that $\F(S)-y+m \notin S$. This implies that $y-m \in K \subseteq \widetilde{I}-M$ by (1) and, thus, $y=(y-m)+m \in \widetilde I$ yields a contradiction. Hence, $\F(I)-x=\F(S)-y \in \PF(S)$.  \\
(5) $\Rightarrow$ (4) Assume by contradiction that there exists $x \in (K-(M-M))\setminus \widetilde I$. It easily follows from the definition that there is $s \in S$ such that $x+s \in \PF(\widetilde{I})$. Then, $\F(S)-x-s \in \PF(S) \cup \{0\} \subseteq M-M$ by (5) and $\F(S)-s=x +(\F(S)-x-s) \in (K-(M-M)) + (M-M) \subseteq K$ gives a contradiction.
\end{proof}

\begin{remark} \rm
{\bf 1.} In \cite[Theorem 2.4]{N} it is proved that a numerical semigroup $S$ is almost symmetric if and only if $\F(S)-f \in \PF(S)$ for every $f \in \PF(S) \setminus \{\F(S)\}$. Hence, the last condition of Proposition \ref{almost canonical ideal} can be considered a generalization of this result. \\
{\bf 2.} Almost canonical ideals naturally arise characterizing the almost symmetry of the numerical duplication $\du$ of $S$ with respect to the ideal $I$ and $b \in S$, a construction introduced in \cite{DS}. Indeed \cite[Theorem 4.3]{DS} says that $\du$ is almost symmetric if and only if $I$ is almost canonical and $K-\widetilde{I}$ is a numerical semigroup. \\
{\bf 3.} Let $T$ be an almost symmetric numerical semigroup with odd Frobenius number (or, equivalently, odd type). Let $b$ be an odd integer such that $2b \in T$ and set $I=\{x \in \mathbb{Z} \mid 2x+b \in T\}$. Then, \cite[Proposition 3.3]{S} says that $T$ can be realized as a numerical duplication $T=\du$, where $S=T/2=\{y \in \mathbb{Z} \mid 2y \in T\}$, while \cite[Theorem 3.7]{S} implies that $I$ is an almost canonical ideal of $S$. In general this is not true if the Frobenius number of $T$ is even. 
\end{remark}

Since $\F(K-(M-M))= \F(\widetilde{I})=\F(K)$ and $\widetilde{I} \subseteq K$ for every relative ideal $I$, Condition (4) of Proposition \ref{almost canonical ideal} allows to find all the almost canonical ideals of a numerical semigroup. Clearly it is enough to focus on the relative ideals with Frobenius number $\F(S)$.


\begin{corollary} \label{Number of almost canonical ideals}
Let $S$ be a numerical semigroup with type $t$.
If $I$ is almost canonical, then $t(I)\leq t+1$. Moreover, for every integer $i$ such that $1 \leq i \leq t+1$, there are exactly $\binom{t}{i-1}$ almost canonical ideals of $S$ with Frobenius number $\F(S)$ and type $i$. In particular, there are exactly $2^{t}$ almost canonical ideals of $S$ with Frobenius number $\F(S)$.
\end{corollary}

\begin{proof}
Let $C=\{s \in S \mid s>\F(S)\}=K-\mathbb{N}$ be the conductor of $S$ and let $n(S)=|\{s \in S \mid s<\F(S)\}|$. It is straightforward to see that $g(S)+n(S)=\F(S)+1$. If $I$ is almost canonical, Proposition \ref{almost canonical ideal} implies that
\begin{align*}
t(I)&=g(I)+g(S)-\F(S)\leq |\mathbb{N}\setminus (K-(M-M))|-n(S)+1= \\
&=|(M-M)\setminus (K-\mathbb{N})|-n(S)+1
=|(M-M)\setminus C|-n(S)+1=\\
&=|(M-M)\setminus S|+|S \setminus C|-n(S)+1=t+n(S)-n(S)+1=t+1.
\end{align*}
By Proposition \ref{almost canonical ideal} an ideal $I$ with $\F(I)=\F(S)$ is almost canonical if and only if $K-(M-M) \subseteq I \subseteq K$ and we notice that $|K \setminus (K-(M-M))|=|(M-M)\setminus S|=t$. Let $A \subseteq (K \setminus (K-(M-M)))$ and consider $I=(K-(M-M)) \cup A$. We claim that $I$ is an ideal of $S$. Indeed, let $x \in A$, $m \in M$ and $y \in (M-M)$. It follows that $m+y \in M$ and, then, $x+m+y \in K$, since $K$ is an ideal. Therefore, $x+m \in K-(M-M)$ and $I$ is an ideal of $S$. Moreover, by \cite[Lemma 4.7]{DS}, $t(I)=|(K-I)\setminus S|+1=|K\setminus I|+1=t+1-|A|$ and the thesis follows, because there are $\binom{t}{i-1}$ subsets of $K\setminus (K-(M-M))$ with cardinality $t+1-i$.
\end{proof}

If $S$ is a symmetric semigroup, the only almost canonical ideals with Frobenius number $\F(S)$ are $M$ and $S$. In this case $t(M)=t(S)+1=2$.
If $S$ is pseudo-symmetric, the four almost canonical ideals with Frobenius number $\F(S)$ are $M$, $S$, $M \cup \{\F(S)/2\}$ and $K$. In this case $t(M)=3$, $t(S)=t(M \cup \{\F(S)/2\})=2$ and $t(K)=1$.

\section{GAS numerical semigroups}

In \cite{CGKM} it is introduced the notion of $n$-almost Gorenstein local rings, briefly $n$-AGL rings, where $n$ is a non-negative integer. These rings generalize almost Gorenstein ones that are obtained when either $n=0$, in which case the ring is Gorenstein, or $n=1$. In particular, in \cite{CGKM} it is studied the case of the 2-AGL rings, that are closer to be almost Gorenstein, see also \cite{GIT}. 

Given a numerical semigroup $S$ with standard canonical ideal $K$ we denote by $\langle K \rangle$ the numerical semigroup generated by $K$. Following \cite{CGKM} we say that $S$ is $n$-AGL if $|\langle K \rangle \setminus K|=n$. It follows that $S$ is symmetric if and only if it is 0-AGL, whereas it is almost symmetric and not symmetric if and only if it is 1-AGL.

It is easy to see that a numerical semigroup is 2-AGL if and only if $2K=3K$ and $|2K\setminus K|=2$, see \cite[Theorem 1.4]{CGKM} for a proof in a more general context. We now give another easy characterization that will lead us to generalize this class.

\begin{proposition}
A numerical semigroup $S$ is 2-AGL if and only if $2K=3K$ and $2K \setminus K=\{\F(S)-x, \F(S)\}$ for a minimal generator $x$ of $S$.
\end{proposition}

\begin{proof}
One implication is trivial, so assume that $S$ is 2-AGL.
Since $S$ is not symmetric, there exists $k \in \mathbb{N}$ such that $k$ and $\F(S)-k$ are in $K$ and so $\F(S) \in 2K \setminus K$. Let now $a \in (2K \setminus K) \setminus \{\F(S)\}$. Since $a \notin K$, we have $\F(S)-a \in S$. Assume that $\F(S)-a=s_1+s_2$ with $s_1,s_2 \in S \setminus \{0\}$. It follows that $\F(S)-s_1=a+s_2 \in 2K$, since $2K$ is a relative ideal, and by definition $\F(S)-s_1 \notin K$. Therefore, $\{a,\F(S)-s_1,\F(S)\} \subseteq 2K \setminus K$ and this is a contradiction, since $S$ is 2-AGL. Hence, $a=\F(S)-x$, where $x$ is a minimal generator of $S$.
\end{proof}

In light of the previous proposition we propose the following definition.

\begin{definition} \rm 
We say that $S$ is a {\it generalized almost symmetric} numerical semigroup, briefly {\rm GAS} numerical semigroup, if either $2K=K$ or $2K \setminus K=\{\F(S)-x_1, \dots, \F(S)-x_r, \F(S)\}$ for some $r \geq 0$ and some minimal generators $x_1, \dots, x_r$ of $S$ such that $x_i-x_j \notin \PF(S)$ for every $i,j$.
\end{definition}

The last condition could seem less natural, but these semigroups have a better behaviour. For instance, in Theorem \ref{Livelli più alti} we will see that this condition ensures that every element in $ \langle K \rangle \setminus K$ can be written as $\F(S)-x$ for a minimal generator $x$ of $S$. 

We recall that $S$ is symmetric if and only if $2K=K$ and it is almost symmetric exactly when $2K \setminus K \subseteq \{\F(S)\}$.

\begin{examples} \rm
{\bf 1.} Let $S= \langle 9, 24, 39, 43, 77 \rangle$. Then, $\PF(S)=\{58, 73, 92, 107\}$ and $2K \setminus K=\{107-77,107-43,107-39,107-24,107-9,107\}$. Hence, $S$ is a GAS semigroup. \\
{\bf 2.} If $S=\langle 7,9,15 \rangle$, we have $2K=3K$ and $2K \setminus K=\{26-14,26-7,26\}$. Hence, $S$ is 3-AGL but it is not GAS because $14$ is not a minimal generator of $S$. \\
{\bf 3.} Consider the semigroup $S=\langle 8, 11, 14, 15, 17, 18, 20, 21 \rangle$. We have $2K \setminus K=\{13-11,13-8,13\}$, but $S$ is not GAS because $11-8 \in \PF(S)$. In this case $2K=3K$ and thus $S$ is 3-AGL.
\end{examples}

The last example shows that in a numerical semigroup $S$ with maximal embedding dimension there could be many minimal generators $x$ such that $\F(S) -x \in 2K\setminus K$. This is not the case if we assume that $S$ is GAS.

\begin{proposition} \label{MED}
If $S$ has maximal embedding dimension $e$ and it is {\rm GAS}, then it is either almost symmetric or {\rm 2-AGL} with $2K\setminus K=\{\F(S)-e,\F(S)\}$.
\end{proposition}

\begin{proof}
Assume that $S$ is not almost symmetric and let $\F(S)-x=k_1+k_2 \in 2K\setminus K$ with $x\neq 0$ and $k_1,k_2 \in K$. Let $x\neq e$ and consider $\F(S)-e=k_1+k_2+x-e$. Since $x-e \leq \F(S)-e < \F(S)$ and $S$ has maximal embedding dimension, $x-e \in \PF(S) \setminus \{\F(S)\} \subseteq K$ and, therefore, $\F(S)-e \in 3K \setminus K$. Moreover, $\F(S)-e$ cannot be in $2K$, because $S$ is GAS and $x-e \in \PF(S)$, then, $k_1+x-e \in 2K \setminus K$. Hence, we have $\F(S)-(\F(S)-k_1-x+e) \in 2K\setminus K$ and, thus, $\F(S)-k_1-x+e$ is a minimal generator of $S$. Since $S$ has maximal embedding dimension, this implies that $\F(S)-k_1-x \in \PF(S)$ and, then, $\F(S)-k_1 \in S$ yields a contradiction, since $k_1 \in K$. This means that $x=e$ and $2K\setminus K=\{\F(S)-e,\F(S)\}$.

Suppose by contradiction that $2K \neq 3K$ and let $\F(S)-y \in 3K \setminus 2K$. In particular, $\F(S)-y \notin K$ and, therefore, $y \in S$. If $\F(S)-y =k_1+k_2+ k_3$ with $k_i \in K$ for every $i$, then $k_1+k_2 \in 2K \setminus K$ and, thus, $k_1+k_2=\F(S)-e$. This implies that $\F(S)-e<\F(S)-y$, i.e. $y<e$, that is a contradiction.
\end{proof}

In particular, we note that in a 2-AGL semigroup with maximal embedding dimension it always holds that $2K \setminus K=\{\F(S)-e, \F(S)\}$.

\begin{proposition} \label{Characterizations GAS}
Given a numerical semigroup $S$, the following conditions are equivalent:
\begin{enumerate}
\item $S$ is {\rm GAS};
\item $x-y \notin (M-M)$ for every different $x,y \in M\setminus (S-K)$;
\item either $S$ is symmetric or $2M \subseteq S-K \subseteq M$ and $M-M=((S-K)-M) \cup \{0\}$.
\end{enumerate} 
\end{proposition}

\begin{proof}
If $S$ is symmetric, then $M \subseteq S-K$ and both (1) and (2) are true, so we assume $S \neq K$. \\ 
$(1) \Rightarrow (2)$ Note that $K-S=K$ and $K-(S-K)=K-((K-K)-K)=K-(K-2K)=2K$. Thus, $x \in S \setminus (S-K)$ if and only if $\F(S)-x \in (K-(S-K))\setminus (K-S)=2K \setminus K$. Hence, if $S$ is GAS, then $x-y \notin S \cup \PF(S)=M-M$ for every $x,y \in M\setminus (S-K)$. \\
$(2) \Rightarrow (1)$ If $x$, $y \in M \setminus (S-K)$, then $\F(S)-x$, $\F(S)-y \in 2K \setminus K$ and $x-y \notin \PF(S)$, since it is not in $M-M$. We only need to show that $x$ is a minimal generator of $S$. If by contradiction $x=s_1+s_2$, with $s_1$, $s_2 \in M$, it follows that also $s_1$ is in $M \setminus (S-K)$. Therefore, $s_2=x-s_1 \in M$ yields a contradiction since $x-s_1 \notin M-M$ by hypothesis. \\ 
$(2) \Rightarrow (3)$ Since $S$ is not symmetric, $S-K$ is contained in $M$. Moreover, if $2M$ is not in $S-K$, then there exist $m_1, m_2 \in M$ such that $m_1+m_2 \in 2M \setminus (S-K)$. Clearly also $m_1$ is not in $S-K$ and $(m_1+m_2)-m_1=m_2 \in M \subseteq M-M$ yields a contradiction. 

It always holds that $((S-K)-M) \cup \{0\} \subseteq M-M$, then given $x \in (M-M) \setminus \{0\}$ and $m \in M$, we only need to prove that $x+m \in S-K$. If $m \in M \setminus (S-K)$ and $x+m \notin S-K$, then $(x+m)-m=x \in M-M$ gives a contradiction.
If $m \in (S-K) \setminus 2M$ and $k \in K$, then $0 \neq m+k \in S$ and, so, $x+m+k \in M$, that implies $x+m \in S-K$.
Finally, if $m \in 2M$, then $x+m \in 2M \subseteq S-K$. \\
$(3) \Rightarrow (2)$ Let $x,y \in M \setminus (S-K)$ with $x \neq y$ and assume by contradiction that $x-y \in (M-M)=((S-K)-M)\cup \{0\}$. By hypothesis $y \in M$, then $x=(x-y)+y \in S-K$ yields a contradiction.
\end{proof}

In the definition of GAS semigroup we required that in $2K\setminus K$ there are only elements of the type $\F(S)-x$ with $x$ minimal generator of $S$. In general, this does not imply that the elements in $3K\setminus 2K$ are of the same type. For instance, consider $S=\langle 8,12,17,21,26,27,30,31 \rangle$, where $2K \setminus K=\{23-21,23-17,23-12,23-8,23\}$ and $3K\setminus 2K=\{23-20,23-16\}$. However, by Proposition \ref{MED}, this semigroup is not GAS. In fact, this never happens in a GAS semigroup as we are going to show in Theorem \ref{Livelli più alti}. First we need a lemma.

\begin{lemma} \label{Lemma livelli più alti}
Assume that $2K \setminus K=\{\F(S)-x_1, \dots, \F(S)-x_r, \F(S)\}$ with $x_1, \dots, x_r$ minimal generators of $S$. If $\F(S)-x \in nK \setminus (n-1)K$ for some $n>2$ and $x=s_1+s_2$ with $s_1$, $s_2 \in M$, then $\F(S)-s_1 \in (n-1)K$.
\end{lemma}

\begin{proof}
Let $\F(S)-(s_1+s_2)=k_1 + \dots + k_n \in nK \setminus (n-1)K$ with $k_i \in K$ for $1 \leq i \leq n$. Since $\F(S)-(s_1+s_2) \notin (n-1)K$, we have $\F(S) \neq k_1+k_2 \in 2K \setminus K$ and, then, $\F(S)-(k_1+k_2)$ is a minimal generator of $S$. Since $\F(S)-(k_1+k_2)=s_1+s_2+k_3+ \dots + k_{n}$, this implies that $s_1+k_3 + \dots + k_{n}\notin S$, that is $k_1+k_2+s_2 =\F(S)-(s_1+k_3+\dots + k_{n}) \in K$. Therefore, $\F(S)-s_1=(k_1+k_2+s_2)+k_3+\dots+ k_{n} \in (n-1)K$ and the thesis follows.
\end{proof}

\begin{theorem} \label{Livelli più alti}
Let $S$ be a {\rm GAS} numerical semigroup that is not symmetric. Then, $\langle K \rangle \setminus K=\{\F(S)-x_1, \dots, \F(S)-x_r, \F(S)\}$ for some minimal generators $x_1, \dots, x_r$ with $r \geq 0$ and $x_i-x_j \notin \PF(S)$ for every $i$ and $j$.
\end{theorem}

\begin{proof}
We first prove that $x_i-x_j \notin \PF(S)$ for every $i$ and $j$ without assuming that $x_i$ and $x_j$ are minimal generators. We can suppose that $x_i=x_1$ and $x_j=x_2$.

Let $\F(S)-x_1=k_1+\dots +k_n \in nK \setminus (n-1)K$ with $k_i \in K$ for every $i$ and assume by contradiction that $x_1-x_2 \in \PF(S)$. 
We note that $\F(S)-x_2=k_1+\dots +k_n + (x_1-x_2)$ and $k_1+(x_1-x_2) \in K$. Indeed, if $\F(S)-k_1 -(x_1-x_2)=s \in S$, then $s \neq 0$ and $\F(S)-k_1=(x_1-x_2)+s \in S$ yields a contradiction. If $k_1+k_2+(x_1-x_2) \notin K$, then it is in $2K \setminus K$ and, since also $k_1+k_2 \in 2K \setminus K$, we get a contradiction because their difference is a pseudo-Frobenius number. Hence, $k_1+k_2+(x_1-x_2) \in K$.

We proceed by induction on $n$. If $n=2$, it follows that $\F(S)-x_2=k_1+k_2+(x_1-x_2) \in K$, that is a contradiction. So, let $n \geq 3$ and let $i$ be the minimum index for which $k_1+ \dots + k_i + (x_1-x_2) \notin K$. It follows that $k_1+ \dots + k_i + (x_1-x_2) \in 2K\setminus K$ and, since also $k_1+k_2 \in 2K \setminus K$, this implies that $k_3+ \dots + k_i +(x_1-x_2) \notin \PF(S)$. Moreover, it cannot be in $S$, because it is the difference of two minimal generators, since $S$ is GAS. Therefore, there exists $m \in M$ such that $k_3+ \dots + k_i +(x_1-x_2)+ m \notin S$, that means $\F(S)-(k_3+ \dots + k_i +(x_1-x_2)+m)=k' \in K$. Thus, $\F(S)-((x_1-x_2)+m)=k'+k_3+ \dots + k_i \in jK \setminus K$ for some $1< j < n$. Moreover, $\F(S)-m=k'+k_3+ \dots + k_i +(x_1-x_2)\in \langle K \rangle \setminus K$ and by induction $(x_1-x_2)+m-m \notin \PF(S)$, that is a contradiction. Hence, $x_1-x_2 \notin \PF(S)$.

Let now $h \geq 3$. To prove the theorem it is enough to show that, if $\F(S)-x \in hK \setminus (h-1)K$, then $x$ is a minimal generators of $S$. We proceed by induction on $h$.
Using the GAS hypothesis, the case $h=3$ is very similar to the general case, so we omit it (the difference is that also $\F(S) \in 2K \setminus K$). Suppose by contradiction that $x=s_1+s_2$ and $\F(S)-(s_1+s_2)=k_1+ \dots +k_h \in hK \setminus (h-1)K$ with $k_1, \dots, k_h \in K$ and $s_1, s_2 \in M$. 
Clearly, $\F(S)-s_1 \notin K$ and by Lemma \ref{Lemma livelli più alti} we have $\F(S)-s_1 \in (h-1)K$; in particular, $s_1$ is a minimal generator of $S$ by induction. Let $1< i < h$ be such that $\F(S)-s_1 \in iK \setminus (i-1)K$. Since $\F(S)-(s_1+s_2) \notin (h-1)K$, we have $k_1+ \dots + k_i \in iK \setminus (i-1)K$ and, by induction, $\F(S)-(k_1+\dots+k_i)$ is a minimal generator of $S$ and $\F(S)-(k_1+\dots+k_i)-s_1 \notin \PF(S)$ by the first part of the proof. This means that there exists $s \in M$ such that $\F(S)-(k_1+\dots+k_i)-s_1+s \notin S$, i.e. $k_1+\dots+k_i+s_1-s \in K$. This implies that $\F(S)-(s_2+s)=(k_1 + \dots + k_i +s_1-s)+k_{i+1}+\dots+k_h \in (h-i+1)K$ and, since $h-i+1 <h$, the induction hypothesis yields a contradiction because $s_2+s$ is not a minimal generator of $S$.
\end{proof}

We recall that in an almost symmetric numerical semigroup $\F(S)-f \in \PF(S)$ for every $f \in \PF(S)\setminus \{\F(S)\}$, see \cite[Theorem 2.4]{N}. The following proposition generalizes this fact.

\begin{proposition} \label{PF GAS}
Let $S$ be a numerical semigroup with $2K \setminus K=\{\F(S)-x_1, \dots, \F(S)-x_r, \F(S)\}$, where $x_i$ is a minimal generator of $S$ for every $i$. 
\begin{enumerate}
	\item For every $i$, there exist $f_j, f_k \in \PF(S)$ such that $f_j+f_k=\F(S)+x_i$.
	\item For every $f \in \PF(S)\setminus \{\F(S)\}$, it holds either $\F(S)-f \in \PF(S)$ or $\F(S)-f+x_i \in \PF(S)$ for some $i$.
\end{enumerate}
\end{proposition}

\begin{proof}
Let $\F(S)-x_i=k_1+k_2 \in 2K\setminus K$ for some $k_1, k_2 \in K$ and let $s \in M$. 
Since $x_i+s \in S$, we have $\F(S)-x_i-s \notin K$ and then $\F(S)-x_i-s=k_1+k_2-s \notin 2K$ because $x_i+s$ is not a generator of $S$. In particular, $k_1-s$ and $k_2-s$ are not in $K$. This means that $\F(S)-k_1+s$ and $\F(S)-k_2+s$ are in $S$ and, thus, $\F(S)-k_1, \F(S)-k_2 \in \PF(S)$. Moreover, $\F(S)-k_1+\F(S)-k_2=2\F(S)-(\F(S)-x_i)=\F(S)+x_i$ and (1) holds.

Let now $f \in \PF(S) \setminus \{\F(S)\}$ and assume that $\F(S)-f \notin \PF(S)$. Then, there exists $s \in M$ such that $\F(S)-f+s \in \PF(S)$. In particular, $f-s \in K$ and $\F(S)-s=(\F(S)-f)+(f-s) \in 2K \setminus K$; thus, $s$ has to be equal to $x_i$ for some $i$ and $\F(S)-f+x_i \in \PF(S)$.
\end{proof}

\begin{examples} \rm \label{Examples}
{\bf 1.} Let $S=\langle 28,40,63,79,88\rangle$. We have $2K \setminus K=\{281-28,281\}$ and $S$ is 2-AGL. In this case $\PF(S)=\{100,132,177,209,281\}$ and $100+209=132+177=281+28$. \\
{\bf 2.} Consider $S= \langle 67, 69, 76, 78, 86 \rangle$. Here $2K \setminus K=\{485-86,485\}$ and the semigroup is 2-AGL. Moreover, $\PF(S)=\{218, 226, 249, 259, 267, 322, 485	\}$, $218+267=226+259=485$ and $249+322=485+86$. \\
{\bf 3.} If $S=\langle 9,10,12,13 \rangle$, then $2K \setminus K=\{17-13,17-12,17-10,17-9,17\}$ and $\PF(S)=\{11,14,15,16,17\}$. Hence, $S$ is GAS and, according to the previous proposition, we have
\begin{align*}
\F(S)+9&=11+15   &\F(S)+12=14+15&\\
\F(S)+10&=11+16  &\F(S)+13=14+16&.
\end{align*}
{\bf 4.} Conditions (1) and (2) in Proposition \ref{PF GAS} do not imply that every $x_i$ is a minimal generator. For instance, if we consider the numerical semigroup $S=\{15,16,19,20,24\}$, we have $2K \setminus K=\{42-40,42-36,42-32,42-24,42-20,42-19,42-16,42-15,42\}$ and $\PF(S)=\{28,29,33,37,41,42\}$. Moreover,
\begin{align*}
\F(S)+40&=41+41  &\F(S)+20=29+33 \\
\F(S)+36&=37+41  &\F(S)+19=28+33 \\
\F(S)+32&=37+37  &\F(S)+16=29+29 \\
\F(S)+24&=33+33  &\F(S)+15=28+29
\end{align*}
and, so, it is straightforward to see that the conditions in Proposition \ref{PF GAS} hold, but $32$, $36$ and $40$ are not minimal generators.
\end{examples}

We recall that $\L(S)$ denotes the set of the gaps of the second type of $S$, i.e. the integers $x$ such that $x \notin S$ and $\F(S)-x \notin S$, i.e. $x \in K \setminus S$, and that $S$ is almost symmetric if and only if $\L(S) \subseteq \PF(S)$, see \cite{BF}.

\begin{lemma} \label{Lemma L(S)}
Let $S$ be a numerical semigroup with $2K \setminus K=\{\F(S)-x_1, \dots, \F(S)-x_r,\F(S)\}$, where $x_i$ is a minimal generator of $S$ for every $i$. If $x \in \L(S)$ and $\F(S)-x \notin \PF(S)$, then both $x$ and $\F(S)-x+x_i$ are pseudo-Frobenius numbers of $S$ for some $i$.
\end{lemma}

\begin{proof}
Assume by contradiction that $x \notin \PF(S)$. Therefore, there exists $s \in M$ such that $x+s \notin S$ and, then, $\F(S)-x-s \in K$. Moreover, since $\F(S)-x \notin \PF(S)$, there exists $t \in M$ such that $\F(S)-x+t \notin S$ and then $x-t \in K$. Consequently, $\F(S)-s-t=(\F(S)-x-s)+(x-t) \in 2K$ and $\F(S)-s-t\notin K$, since $s+t \in S$. This is a contradiction, because $s+t$ is not a minimal generator of $S$. Hence, $x \in \PF(S)$ and, since $\F(S)-x \notin \PF(S)$, Proposition \ref{PF GAS} implies that $\F(S)-x+x_i \in \PF(S)$ for some $i$.
\end{proof}

\begin{lemma} \label{difference} As ideal of $M-M$, it holds $\widetilde{M-e}=M-e$ and
\[
K(M-M) \setminus (M-e) =\{x-e \mid x \in \L(S) \text{ and } \F(S) - x \notin \PF(S)\}.
\]
\end{lemma}

\begin{proof}
We notice that $\F(S)-e \notin (M-M)$ and, if $y > \F(S)-e$ and $m \in M$, we have $y+m >\F(S)-e+m \geq \F(S)$. Therefore, $\F(M-M)=\F(S)-e=\F(M-e)$ and, then, $\widetilde{M-e}=M-e$.

We have $x-e \in K(M-M) \setminus (M-e)$ if and only if $x\notin M$ and $(\F(S)-e)-(x-e) \notin (M-M)$ that is in turn equivalent to $x \notin M$ and $\F(S)-x \notin S \cup \PF(S)$. Since $x \neq 0$, this means that $x \in \L(S)$ and $\F(S)-x \notin \PF(S)$. 
\end{proof}

The following corollary was proved in \cite[Theorem 5.2]{B} in a different way.

\begin{corollary} \label{canonical ideal}
$S$ is almost symmetric if and only if $M-e$ is a canonical ideal of $M-M$.
\end{corollary}

\begin{proof}
By definition $M-e$ is a canonical ideal of $M-M$ if and only if $K(M-M) = (M-e)$. In light of the previous lemma, this means that there are no $x \in \L(S)$ such that $\F(S)-x \notin \PF(S)$, that is equivalent to say that $\L(S)\subseteq \PF(S)$, i.e. $S$ is almost symmetric.
\end{proof}

In \cite[Corollary 8]{BF} it was first proved that $S$ is almost symmetric with maximal embedding dimension if and only if $M-M$ is a symmetric semigroup. In general it holds $M-M \subseteq M-e \subseteq K(M-M)$ and the first inclusion is an equality if and only if $S$ has maximal embedding dimension, whereas the previous corollary says that the second one is an equality if and only if $S$ is almost symmetric.
Moreover, if $S$ has maximal embedding dimension, in \cite[Corollary 5.4]{CGKM} it is proved that $S$ is 2-AGL if and only if $M-M$ is an almost symmetric semigroup which is not symmetric. If we want to generalize this result in the same spirit of Corollary \ref{canonical ideal}, it is not enough to consider the 2-AGL semigroups, but we need that $S$ is GAS. More precisely, we have the following result.

\begin{theorem} \label{T. Almost Canonical ideal of M-M}
The semigroup $S$ is {\rm GAS} if and only if $M-e$ is an almost canonical ideal of the semigroup $M-M$.
\end{theorem}

\begin{proof}
In the light of Remark \ref{Rem as}.4 and Lemma \ref{difference}, $M-e$ is an almost canonical ideal of $M-M$ if and only if
\begin{equation} \label{Eq.Canonical Ideal of M-M}
K(M-M) \setminus (M-e) \subseteq ((M-e)-((M-M)\setminus \{0\})).
\end{equation}
Assume that $S$ is GAS with $2K \setminus K=\{\F(S)-x_1, \dots, \F(S)-x_r, \F(S)\}$.
By Lemma \ref{difference} the elements of $K(M-M) \setminus (M-e)$ can be written as $x-e$ with $x \in \L(S)$ and $\F(S)-x \notin \PF(S)$. In addition, Lemma \ref{Lemma L(S)} implies that both $x$ and $\F(S)-x+x_i$ are pseudo-Frobenius numbers of $S$ for some $i$.
Let $0 \neq z\in (M-M)$. We need to show that $x-e+z\in M-e$, i.e. $x+z \in M$. Assume by contradiction $x+z \notin M$, which implies $\F(S)-x-z \in K$. Since $x+z \notin M$ and $x \in \PF(S)$, it follows that $z \notin M$ and, then, $z \in \PF(S)$; hence, $z+x_i \in M$ and $\F(S)-z-x_i \notin K$. We also have that $x-x_i \in K$, since $\F(S)-x+x_i \in \PF(S)$. Therefore, 
\[
\F(S)-z-x_i=(\F(S)-x-z)+(x-x_i) \in 2K \setminus K
\]
and this yields a contradiction because $(z+x_i)-x_i \in \PF(S)$ and $S$ is a GAS semigroup.

Conversely, assume that the inclusion (\ref{Eq.Canonical Ideal of M-M}) holds.
An element in $2K\setminus K$ can be written as $\F(S)-s$ for some $s \in S$, since it is not in $K$. 
Assume by contradiction that $s \neq 0$ is not a minimal generator of $S$, i.e. $\F(S)-s_1-s_2=k_1+k_2 \in 2K\setminus K$ for some $s_1,s_2 \in M$ and $k_1, k_2 \in K$. It follows that $\F(S)-k_1-s_1=k_2 + s_2 \notin S$, otherwise $\F(S)-s_1 \in K$. Moreover, $k_1+s_1 \notin \PF(S) \cup S$, since $k_1+s_1+s_2 = \F(S)-k_2 \notin S$.
Hence, Lemma \ref{difference} and our hypothesis imply that 
\[k_2+s_2-e=\F(S)-k_1-s_1-e \in ((M-e)-((M-M)\setminus \{0\})).\]
Therefore, $\F(S)-k_1-e=(k_2+s_2-e)+s_1 \in M-e$ and, thus, $k_1 \notin K$ yields a contradiction.
This means that $2K \setminus K=\{\F(S)-x_1, \dots, \F(S)-x_r, \F(S)\}$ with $x_i$ minimal generator of $S$ for every $i$. Now, assume by contradiction that $z=x_i-x_j \in \PF(S)$ for some $i,j$ and let 
$\F(S)-x_i=\F(S)-x_j-z=k_1+k_2$ for some $k_1, k_2 \in K$.
Since $k_2+z+x_j=\F(S)-k_1 \notin S$, it follows that $k_2+z \notin S \cup \PF(S)$. Moreover, $\F(S)-k_2-z \notin S$, otherwise $\F(S)-k_2\in S$. Therefore, Lemma \ref{difference} and inclusion (\ref{Eq.Canonical Ideal of M-M}) imply that $\F(S)-k_2-z-e \in ((M-e)-((M-M)\setminus \{0\}))$ and, since $z \in M-M$, it follows that $\F(S)-k_2 \in M$ which is a contradiction because $k_2 \in K$.
\end{proof}

\begin{example} \rm
Consider $S=\langle 9,13,14,15,19 \rangle$, that is a GAS numerical semigroup with $2K \setminus K=\{25-15,25-13,25-9,25\}$. Then, $M-9$ is an almost canonical ideal of $M-M$ by the previous theorem. In fact
\begin{equation*}
\begin{split}
&M-M=\{0,9,13,14,15,17, \rightarrow\}, \\
&K(M-M)=\{0,4,5,6,8,9,10,11,12,13,14,15,17 \rightarrow\},\\
&M-9=\{0,4,5,6,9,10,13,14,15,17 \rightarrow\}, \\
&(M-9)-((M-M)\setminus \{0\})=K(M-M) \cup \{16\}=\{0,4,5,6,8 \rightarrow\}.
\end{split} 
\end{equation*}
\end{example}

\begin{remark} \rm
If $S$ is {\rm GAS}, it is possible to compute the type of $M-e$ seen as an ideal of the semigroup $M-M$. In fact by Theorem \ref{T. Almost Canonical ideal of M-M} and Proposition \ref{almost canonical ideal} it follows that
\begin{align*}
t(M-e)&=g(M-e)+g(M-M)-\F(M-M)= \\
&=g(M)-e+g(S)-t(S)-\F(S)+e=
2g(S)+1-t(S)-\F(S).
\end{align*}
Moreover, we recall that $2g(S) \geq t(S)+\F(S)$ is always true and the equality holds exactly when $S$ is almost symmetric. Therefore,
as $t(S)$ is a measure of how far $S$ is from being symmetric, $t(M-e)=t(M)$ (as ideal of $M-M$) can be seen as a measure of how far $S$ is from being almost symmetric. On the other hand, we note that the type of $M$ as an ideal of $S$ is simply $t(S)+1$.
\end{remark}

If $S$ has type 2 and $\PF(S)=\{f,\F(S)\}$, in \cite[Theorem 6.2]{CGKM} it is proved that $S$ is 2-AGL if and only if $3(\F(S)-f) \in S$ and $\F(S)=2f-x$ for some minimal generator $x$ of $S$. In the next proposition we generalize this result to the GAS case.

\begin{proposition} \label{type 2}
Assume that $S$ is not almost symmetric and that it has type 2, i.e. $\PF(S)=\{f,\F(S)\}$. Then, $S$ is {\rm GAS} if and only if $\F(S)=2f-x$ for some minimal generator $x$ of $S$. In this case, if $n$ is the minimum integer for which $n(\F(S)-f) \in S$, then  $|2K \setminus K|=2$, $|3K \setminus 2K|= \dots = |(n-1)K \setminus (n-2)K|=1$ and $nK=(n-1)K$.
\end{proposition}

\begin{proof}
Assume first that $S$ is GAS and let $\F(S)-x$, $\F(S)-y \in 2K \setminus K$. Proposition \ref{PF GAS} implies that $\F(S)+x=f_1+f_2$ and $\F(S)+y=f_3+f_4$ for some $f_1,f_2,f_3,f_4 \in \PF(S)$. Since $f_i$ has to be different from $\F(S)$ for all $i$, it follows that $\F(S)+x=\F(S)+y=2f$ and, then, $x=y$. In particular, $\F(S)=2f-x$. 

Assume now that $\F(S)=2f-x$ for some minimal generator $x$ of $S$. Clearly, $\F(S)-x=2(\F(S)-f) \in 2K \setminus K$. Let $y \neq 0,x$ be such that $\F(S)-y \in 2K \setminus K$. Since $2K \setminus K$ is finite, we may assume that $y$ is maximal among such elements with respect to $\leq_S$, that is $\F(S)-(y+m) \notin 2K\setminus K$ for every $m \in M$. Let $\F(S)-y=k_1+k_2$ with $k_1$, $k_2 \in K$. Since $\F(S)-y-m=k_1+k_2-m \notin 2K \setminus K$, then $k_1-m$ and $k_2-m$ are not in $K$, which is equivalent to $\F(S)-k_1+m \in S$ and $\F(S)-k_2+m \in S$ for every $m \in M$. This means that $\F(S)-k_1$, $\F(S)-k_2 \in \PF(S)\setminus \{\F(S)\}$ which implies $\F(S)-y=2(\F(S)-f)=\F(S)-x$ and, thus, $x=y$. Therefore, $|2K \setminus K|=2$ and $S$ is GAS. 

Moreover, if $S$ is GAS and $\F(S)-y =k_1+\dots+k_r \in rK \setminus (r-1)K$ with $r>2$ and $k_1, \dots, k_r \in K$, then $k_1= \dots=k_r=\F(S)-f$ because $k_i+k_j \in 2K \setminus K$ for every $i$ and $j$. Therefore, if $n(\F(S)-f) \in S$, then $nK=(n-1)K$. Assume that $r(\F(S)-f) \notin S$. Clearly, it is in $rK$ and we claim that it is not in $K$. In fact, if $r(\F(S)-f) \in K$, it follows that it is in $\L(S)$ and, if $\F(S)-r(\F(S)-f)=f$, then $(r-1)(\F(S)-f)=0 \in S$ yields a contradiction. Therefore, Lemma \ref{Lemma L(S)} implies that $\F(S)-r(\F(S)-f)+x =f$ and, again, $(r-1)(\F(S)-f)=x \in S$ gives a contradiction. This means that $r(\F(S)-f) \in rK \setminus K$. Moreover, if $r(\F(S)-f)=k_1+\dots+ k_{r'} \in r'K\setminus(r'-1)K$ with $1<r'<r$ and $k_1, \dots, k_{r'} \in K$, we get $k_1=\dots=k_{r'}=\F(S)-f$ as above, that is a contradiction. Hence, $|rK \setminus (r-1)K|=1$ for every $1<r<n$.
\end{proof}

\begin{example} \rm \label{GAS tipo 2}
Consider $S=\langle 5,6,7\rangle$. In this case $f=8$ and $\F(S)=9$. Therefore, the equality $\F(S)=2f-7$ implies that $S$ is GAS. With the notation of the previous corollary we have $n=5$ and, in fact, $2K \setminus K=\{2,9\}$, $3K \setminus 2K=\{3\}$ and $4K \setminus 3K=\{4\}$.  
\end{example}

In \cite{HHS} another generalization of almost Gorenstein ring is introduced. More precisely a Cohen-Macaulay local ring admitting a canonical module $\omega$ is said to be {\it nearly Gorenstein} if the trace of $\omega$ contains the maximal ideal. In the case of numerical semigroups it follows from \cite[Lemma 1.1]{HHS} that $S$ is nearly Gorenstein if and only if $M \subseteq K+(S-K)$, see also the arXiv version of \cite{HHS}. It is easy to see that an almost symmetric semigroup is nearly Gorenstein, but in \cite{CGKM} it is noted that a 2-AGL semigroup is never nearly Gorenstein (see also \cite[Remark 3.7]{BS} for an easy proof in the numerical semigroup case). This does not happen for GAS semigroups.

\begin{corollary}
Let $S$ be a {\rm GAS} semigroup, not almost symmetric, with $\PF(S)=\{f,\F(S)\}$. It is nearly Gorenstein if and only if $3f-2\F(S) \in S$.
\end{corollary}

\begin{proof}
We will use the following characterization proved in \cite{MS}: $S$ is nearly Gorenstein if and only if for every minimal generator $y$ of $S$ there exists $g \in \PF(S)$ such that $g+y-g' \in S$ for every $g' \in \PF(S)\setminus \{g\}$.

By Proposition \ref{type 2} it follows that $\F(S)=2f-x$ with $x$ minimal generator of $S$. 
Let $y \neq x$ another minimal generator of $S$ and assume by contradiction that $\F(S)+y-f \notin S$. Therefore, there exists $s \in S$ such that $\F(S)+y-f+s \in \PF(S)$. If it is equal to $\F(S)$, then $f=y+s \in S$ yields a contradiction. If $\F(S)+y-f+s=f$, then $y+s=2f-\F(S)=x$ by Proposition \ref{type 2} and this gives a contradiction, since $x \neq y$ is a minimal generator of $S$. Hence, $\F(S)+y-f \in S$ for every minimal generator $y \neq x$.
On the other hand, $\F(S)+x-f=2f-x+x-f=f \notin S$ and, therefore, $S$ is nearly Gorenstein if and only if $f+x-\F(S)=3f-2\F(S)\in S$.
\end{proof}

\begin{examples} \rm
{\bf 1.} In Example \ref{GAS tipo 2} we have $3f-2\F(S)=6 \in S$ and, then, the semigroup is both GAS and nearly Gorenstein. \\
{\bf 2.} Consider $S=\langle 9,17,67\rangle$ that has $\PF(S)=\{59,109\}$. Since $2*59-109=9$ and $3*59-2*109=-41 \notin S$, the semigroup is GAS but not nearly Gorenstein. \\
{\bf 3.} If $S=\langle 10,11,12,25 \rangle$, we have $\PF(S)=\{38,39\}$ and $2*38-39=37$ is not a minimal generators, thus, $S$ is not GAS. On the other hand, it is straightforward to check that this semigroup is nearly Gorenstein.
\end{examples}

\begin{remark} \rm
In literature there are other two generalizations of almost Gorenstein ring. One is given by the so-called ring with canonical reduction, introduced in \cite{R}, which is a one-dimensional Cohen-Macaulay local ring $(R,\m)$ possessing a canonical ideal $I$ that is a reduction of $\m$. When $R=k[[S]]$ is a numerical semigroup ring, this definition gives a generalization of almost symmetric semigroup and $R$ has a canonical reduction if and only if $e+\F(S)-g \in S$ for every $g \in \mathbb{N} \setminus S$, see \cite[Theorem 3.13]{R}. This notion is unrelated with the one of GAS semigroup, in fact it is easy to see that $S=\langle 4,7,9,10 \rangle$ is GAS and it doesn't have canonical reductions, while $S=\langle 8,9,10,22 \rangle$ is not GAS, but has a canonical reduction.

Another generalization of the notion of almost Gorenstein ring is given by the so-called generalized Gorenstein ring, briefly GGL, introduced in \cite{GIKT,GK}. A Cohen-Macaulay local ring $(R,\mathfrak{m})$ with a canonical module $\omega$ is said to be GGL with respect to $\mathfrak{a}$ if either $R$ is Gorenstein or there exists an exact sequence of $R$-modules
\[
0 \xrightarrow{} R \xrightarrow{\varphi} \omega \xrightarrow{} C \xrightarrow{} 0
\]
where $C$ is an Ulrich module of $R$ with respect to some $\mathfrak m$-primary ideal $\mathfrak a$ and $\varphi \otimes R/\mathfrak a$ is injective. We note that $R$ is almost Gorenstein and not Gorenstein if and only if it is GGL with respect to $\mathfrak m$. Let $S$ be a numerical semigroup and order $\PF(S)=\{f_1,f_2, \dots, f_t=\F(S)\}$ by the usual order in $\mathbb{N}$. Defining a numerical semigroup GGL if its associated ring is GGL, in \cite{T} it is proved a useful characterization: $S$ is GGL if either it is symmetric or the following properties hold:
\begin{enumerate}
	\item there exists $x \in S$ such that $f_i+f_{t-i}=\F(S)+x$ for every $i=1, \dots, \lceil t/2 \rceil$;
	\item $((c-M) \cap S) \setminus c=\{x\}$, where $c=S-\langle K \rangle$.
\end{enumerate}
Using this characterization it is not difficult to see that also this notion is unrelated with the one of GAS semigroup. In fact, the semigroups in Examples \ref{Examples}.2 and \ref{Examples}.3 are GAS but do not satisfy (1), whereas the semigroup $S=\langle 5,9,12 \rangle$ is not GAS by Proposition \ref{type 2}, because $\PF(S)=\{13,16\}$, but it is easy to see that it is GGL with $x=10$.
\end{remark}

\section{Constructing GAS numerical semigroups}

In this section we study the behaviour of the GAS property with respect to some constructions. In this way we will be able to construct many numerical semigroups satisfying this property. 

\subsection{Gluing of numerical semigroups}

Let $S_1=\langle s_1, \dots, s_n \rangle$ and $S_2=\langle t_1, \dots, t_m \rangle$ be two numerical semigroups and assume that $s_1, \dots, s_n$ and $t_1, \dots, t_m$ are minimal generators of $S_1$ and $S_2$ respectively. Let also $a\in S_2$ and $b \in S_1$ be not minimal generators of $S_2$ and $S_1$ respectively and assume $\gcd(a,b)=1$. The numerical semigroup $\langle aS_1,bS_2 \rangle=\langle as_1, \dots, as_n, bt_1, \dots, bt_m \rangle$ is said to be the gluing of $S_1$ and $S_2$ with respect to $a$ and $b$. It is well-known that $as_1, \dots, as_n, bt_1, \dots, bt_m$ are its minimal generators, see \cite[Lemma 9.8]{RG}.
Moreover, the pseudo-Frobenius numbers of $T=\langle aS_1,bS_2 \rangle$ are
\[
\PF(T)=\{af_1+bf_2+ab \mid f_1 \in \PF(S_1), f_2 \in \PF(S_2)\},
\]
see \cite[Proposition 6.6]{N}. In particular, $t(T)=t(S_1)t(S_2)$ and $\F(T)=a\F(S_1)+b\F(S_2)+ab$. Consequently, since $K(T)$ is generated by the elements $\F(T)-f$ with $f \in \PF(T)$, it is easy to see that $K(T)=\{ak_1+bk_2 \mid k_1 \in K(S_1), k_2 \in K(S_2) \}$.

Since $t(T)=t(S_1)t(S_2)$, it follows that $T$ is symmetric if and only if both $S_1$ and $S_2$ are symmetric, so in the next theorem we exclude this case.

\begin{theorem} \label{gluing}
Let $T$ be a gluing of two numerical semigroups and assume that $T$ is not symmetric. The following are equivalent:
\begin{enumerate}
\item $T$ is {\rm GAS};
\item $T$ is {\rm 2-AGL};
\item $T=\langle 2S, b \mathbb{N} \rangle$ with $b \in S$ odd and $S$ is an almost symmetric semigroup, but not symmetric.
\end{enumerate}
\end{theorem}

\begin{proof}
(2) $\Rightarrow$ (1) True by definition. \\
(1) $\Rightarrow$ (3) Let $T=\langle aS_1, bS_2 \rangle$. Since $T$ is not symmetric, we can assume that $S_1$ is not symmetric and, then, $\F(S_1)=k_1+k_2$ for some $k_1$, $k_2 \in K(S_1)$. This implies that 
\[
\F(T)-b(\F(S_2)+a)=a\F(S_1)+b\F(S_2)+ab-b\F(S_2)-ab=ak_1+ak_2 \in 2K(T) \setminus K(T)
\]
because $\F(S_2)+a \in S_2$. Therefore, since $T$ is GAS, $\F(S_2)+a$ is a minimal generator of $S_2$. By definition of gluing, $a$ is not a minimal generator of $S_2$, so write $a=s+s'$ with $s$, $s' \in M(S_2)$. Since $\F(S_2)+s+s'$ is a minimal generator of $S_2$, we get $\F(S_2)+s=\F(S_2)+s'=0$, i.e. $\F(S_2)=-1$ and $a=s+s'=2$. This proves that $T=\langle 2S_1, b \mathbb{N} \rangle$. Clearly, $b$ is odd by definition of gluing, so we only need to prove that $S_1$ is almost symmetric. Assume by contradiction that it is not almost symmetric and let $s \in M(S_1)$ such that $\F(S_1)-s=k_1+k_2 \in 2K(S_1)\setminus K(S_1)$ with $k_1$, $k_2 \in K(S_1)$. Then
\[
\F(T)-(2s+b)=2\F(S_1)-b+2b-2s-b=2k_1+2k_2 \in 2K(T) \setminus K(T)
\]
and $2s+b$ is not a minimal generator of $T$, contradiction.
\\
(3) $\Rightarrow$ (2) Since $S$ is not symmetric, $\langle K(S) \rangle \setminus K(S)= 2K(S) \setminus K(S)=\{\F(S)\}$. Consider an element $z \in \langle K(T) \rangle \setminus K(T)$, that is $z=2k_1+b\lambda_1 + \dots  + 2k_r + b\lambda_r = 2(k_1+\dots +k_r)+b(\lambda_1+ \dots +\lambda_r)$ for some $k_1, \dots, k_r \in K(S)$ and $\lambda_1, \dots, \lambda_r \in \mathbb{N}$. Since $z \notin K(T)$, then $k_1+ \dots +k_r \notin K(S)$ and so $k_1+\dots +k_r=\F(S)$. Therefore, $z=2\F(S)+b(\lambda_1+\dots + \lambda_r) \in 2K(T)\setminus K(T)$ and, since it is not in $K(T)$ and $\F(T)=2\F(S)+b$, it follows that either $z=2\F(S)$ or $z=2\F(S)+b$. Hence, $|\langle K(T) \rangle \setminus K(T)|=2$ and thus $T$ is 2-AGL.  
\end{proof}

\subsection{Numerical Duplication}

In the previous subsection we have shown that if a non-symmetric GAS semigroup is a gluing, then it can be written as $\langle 2S, b \mathbb{N}\rangle$. This kind of gluing can be seen as a particular case of another construction, the {\it numerical duplication}, introduced in \cite{DS}.

Given a numerical semigroup $S$, a relative ideal $I$ of $S$ and an odd integer $b \in S$, the numerical duplication of $S$ with respect to $I$ and $b$ is defined as $\du=2\cdot S \cup \{2 \cdot I +b\}$, where $2\cdot X=\{2x \mid x\in X\}$ for every set $X$. This is a numerical semigroup if and only if $I+I+b \subseteq S$. This is always true if $I$ is an ideal of $S$ and, since in the rest of the subsection $I$ will always be an ideal, we ignore this condition. In this case, if $S$ and $I$ are minimally generated by $\{s_1, \dots, s_\nu\}$ and $\{i_1, \dots, i_\mu\}$ respectively, then $\du=\langle 2s_1, \dots, 2s_\nu, 2i_1+b, \dots, 2i_\mu+b \rangle$ and these generators are minimal. It follows that $\langle 2S, b \mathbb{N} \rangle = S\! \Join^b \!S$.

\begin{remark} \label{PF duplication} \rm
The Frobenius number of $\du$ is equal to $2\F(I)+b$. Moreover, the odd pseudo-Frobenius numbers of $\du$ are $\{2\lambda+b \mid \lambda \in \PF(I)\}$, whereas the even elements in $\PF(\du)$ are exactly the doubles of the elements in $((M-M) \cap (I-I)) \setminus S$; see the proof of \cite[Proposition 3.5]{DS}. In particular, if $2f \in \PF(\du)$, then $f \in \PF(S)$.
\end{remark}

In this subsection we write $K$ in place of $K(S)$. We note that $S-\langle K \rangle \subseteq S$ and $\F(S-\langle K \rangle)=\F(S)$. 

\begin{lemma} \label{Lemma Numerical Duplication}
Let $S$ be a numerical semigroup, $b \in S$ be an odd integer, $I$ be an ideal of $S$ with $\F(I)=\F(S)$ and $T=\du$. The following hold:
\begin{enumerate}
\item If $k\in K$, then both $2k$ and $2k+b$ are in $K(T)$. In particular, if $\F(S)-x \in iK \setminus K$, then $\F(T)-2x \in iK(T)\setminus K(T)$;
\item Let $k \in K(T)$. If $k$ is odd, then $\frac{k-b}{2} \in K$, otherwise $\F(S)-\frac{k}{2} \notin I$;
\item If $I=S-\langle K \rangle$ and $k \in K(T)$ is even, then $\frac{k}{2} \in jK$ for some $j \geq 1$. 
\item Let $I=S-\langle K \rangle$. If $\F(T)-2i-b \in \langle K(T) \rangle \setminus K(T)$, then $\F(S)-i \in \langle K \rangle \setminus K$ for every $i \in I$. Moreover, $\F(S)-x \in \langle K \rangle \setminus K$ if and only if $\F(T)-2x \in \langle K(T) \rangle \setminus K(T)$.
\end{enumerate}
\end{lemma}

\begin{proof}
(1) If $k \in K$, then $2k+b\in K(T)$, since $\F(T)-(2k+b) = 2(\F(S)-k)\notin 2 \cdot S$. Moreover, $\F(T)-2k=2(\F(S)-k)+b$ and $\F(S)-k \notin I$ because it is not in $S$, so $2k \in K(T)$. Therefore, if $\F(S)-x=k_1+\dots + k_i \in iK \setminus K$ with $k_1, \dots, k_i \in K$, then $\F(T)-2x=2k_1+ \dots +2k_{i-1}+(2k_i+b) \in iK(T)$ and, clearly, it is not in $K(T)$, since $2x \in T$. \\
(2) Let $k$ be odd. Since $2(\F(S)-\frac{k-b}{2})=2\F(S)+b-k=\F(T)-k \notin T$, it follows that $\F(S)-\frac{k-b}{2}\notin S$, i.e. $\frac{k-b}{2} \in K$. If $k$ is even, then $2(\F(S)-\frac{k}{2})+b=\F(T)-k \notin T$ and, thus, $\F(S)-\frac{k}{2}\notin I$.\\
(3) Since $\F(S)-\frac{k}{2} \notin S-\langle K \rangle$ by (2), there exist $i\geq 1$ and $a \in iK$ such that $\F(S)-\frac{k}{2}+a \notin S$, that is $\frac{k}{2}-a \in K$. Hence, $\frac{k}{2}=a+ (\frac{k}{2}-a) \in (i+1)K$. \\
(4) If $\F(T)-2i-b=k_1+ \dots + k_j + \dots k_n \in \langle K(T) \rangle \setminus K(T)$ with $k_1, \dots, k_j \in K(T)$ even and $k_{j+1}, \dots, k_n \in K(T)$ odd, then $\F(S)-i=\frac{k_1}{2}+\dots + \frac{k_j}{2} + \frac{k_{j+1}-b}{2} + \dots + \frac{k_{n}-b}{2} + \frac{(n-j)}{2}b \in \langle K \rangle \setminus K$ by (2) and (3). Using (1) the other statement is analogous.
\end{proof}

\begin{example} \rm \label{Example Numerical Duplication}
{\bf 1.} In the previous lemma we cannot remove the hypothesis $\F(I)=\F(S)$. For instance, consider $S=\langle 3,10,11 \rangle$, $I=\langle 3,10 \rangle$ and $T=S \! \Join^3 \! I $. Then, $\F(I)=11\neq 8=\F(S)$ and we have $\F(S)-6 \in 2K \setminus K$, but $\F(T)-12 \notin \langle K(T) \rangle$. \\
{\bf 2.} In the third statement of the previous lemma, $j$ may be bigger than 1. For instance, consider $S=\langle 6,28,47,97\rangle$ and $T=S\! \Join^{47} \!(S-\langle K\rangle)=\langle 12,56,71,94,115,153,159,194,197,241 \rangle$. Then $88,126,170,182 \in K(T)$, while $44,63,91 \in 2K \setminus K$ and $85 \in 3K \setminus 2K$. 
\end{example}

\begin{corollary} \label{Numerical duplication 2-AGL}
Let $b \in S$ be odd and let $I=S-\langle K \rangle$. The following hold: 
\begin{enumerate}
\item If $S$ is not almost symmetric, then $S\! \Join^b \!M$ is not {\rm GAS};
\item $S$ is n-{\rm AGL} if and only if $\du$ is n-{\rm AGL}.
\end{enumerate}
\end{corollary}

\begin{proof}
(1) Let $T=S\! \Join^b \!M$ and let $x \neq 0$ be such that $\F(S)-x \in 2K \setminus K$. By Lemma \ref{Lemma Numerical Duplication} (1), $\F(T)-2x$ and $\F(T)-(2x+b)$ are in $2K(T) \setminus K(T)$. Even though  $2x+b$ and $2x$ are minimal generators, their difference $b$ is a pseudo-Frobenius number of $T$ by Remark \ref{PF duplication}, because $0 \in \PF(M)$, hence $T$ is not GAS.  \\
(2) Let $T=\du$. By Lemma \ref{Lemma Numerical Duplication} (4) we have that $\F(S)-x \in \langle K\rangle \setminus K$ if and only if $\F(T)-2x \in \langle K(T) \rangle \setminus K(T)$. Moreover, 
if $\F(T)-(2i+b) \in  \langle K(T) \rangle \setminus K(T)$, Lemma \ref{Lemma Numerical Duplication} (4) implies that $\F(S)-i \in \langle K \rangle$ and, since $i \in (S-\langle K \rangle)$, it follows that $\F(S) \in S$, that is a contradiction. Hence, $S$ is $n$-AGL if and only if $T$ is $n$-AGL.
\end{proof}

\begin{remark} \rm
If $S$ is almost symmetric with type $t$, then $M=K-(M-M)$ and, consequently, $S\! \Join^b \!M$ is almost symmetric with type $2t+1$ by \cite[Theorem 4.3 and Proposition 4.8]{DS}.
\end{remark}

If $R$ is a one-dimensional Cohen-Macaulay local ring with a canonical module $\omega$ such that $R \subseteq \omega \subseteq \overline{R}$, in \cite[Theorem 4.2]{CGKM} it is proved that the idealization $R \ltimes (R:R[\omega])$ is 2-AGL if and only if $R$ is 2-AGL. The numerical duplication may be considered the analogous of the idealization in the numerical semigroup case, since they are both members of a family of rings that share many properties (see \cite{BDS}); therefore, Corollary \ref{Numerical duplication 2-AGL} (2) should not be surprising. In the following proposition we generalize this result for the GAS property.

\begin{theorem} \label{Numerical duplication S-<K>}
Let $S$ be a numerical semigroup, let $b \in S$ be an odd integer and let $I=S-\langle K \rangle$. The semigroup $T=\du$ is {\rm GAS} if and only if $S$ is {\rm GAS}.
\end{theorem}

\begin{proof}
Assume that $T$ is GAS and let $\F(S)-x \in 2K \setminus K$. By Lemma \ref{Lemma Numerical Duplication}, $\F(T)-2x\in 2K(T) \setminus K(T)$, so $2x$ is a minimal generator of $T$ and, thus, $x$ is a minimal generator of $S$. Now let $\F(S)-x$, $\F(S)-y \in 2K \setminus K$ and assume by contradiction that $x-y \in \PF(S)$. 
In particular, $S$ is not symmetric and, then, $I=M-\langle K \rangle$. Moreover, $\F(T)-2x$ and $\F(T)-2y$ are in $2K(T) \setminus K(T)$. We also notice that $x-y \in I-I$, indeed, if $i \in I$ and $a \in \langle K \rangle$, it follows that $(x-y)+i+a \in (x-y)+M \subseteq S$. Therefore, Remark \ref{PF duplication} implies that $2(x-y) \in \PF(T)$; contradiction.

Conversely, assume that $S$ is GAS and let $\F(T)-z=k_1+k_2 \in  2K(T) \setminus K(T)$ with $k_1$, $k_2 \in K(T)$. If $z=2i+b$ is odd and both $k_1$ and $k_2$ are odd, then $i\in I$ and $\F(S)-i=(k_1-b)/2+(k_2-b)/2+b \in 2K$ by Lemma \ref{Lemma Numerical Duplication}.(2); on the other hand, if $k_1$ and $k_2$ are both even, $\F(S)-i=k_1/2+k_2/2 \in \langle K \rangle$ by Lemma \ref{Lemma Numerical Duplication}.3. Since $i \in (S-\langle K \rangle)$, in both cases we get $\F(S) \in S$, that is a contradiction. Hence, $z=2x$ is even. If $k_1$ is even and $k_2$ is odd, Lemma \ref{Lemma Numerical Duplication} implies that $\F(S)-x=k_1/2 + (k_2-b)/2 \in (j+1)K \setminus K$ for some $j\geq 1$ and, therefore, by Theorem \ref{Livelli più alti} it follows that $x$ is a minimal generator of $S$, i.e. $z=2x$ is a minimal generator of $T$. Moreover, let $\F(T)-2x$, $\F(T)-2y \in 2K(T)\setminus K(T)$ and assume by contradiction that $2x-2y \in \PF(T)$. Remark \ref{PF duplication} implies that $x-y \in \PF(S) \subseteq K \cup \{\F(S)\}$. Thus, if $\F(T)-2x=k_1+k_2$ with $k_1$, $k_2 \in K(T)$ and $k_1$ even, then $\F(S)-x=k_1/2+(k_2-b)/2 \in \langle K(S) \rangle \setminus K(S)$ by Lemma \ref{Lemma Numerical Duplication} and, so, $\F(S)-y=k_1/2+(k_2-b)/2+(x-y) \in \langle K(S) \rangle \setminus K(S)$. Hence, Theorem \ref{Livelli più alti} yields a contradiction, because $x-y \in \PF(S)$.
\end{proof}

\begin{example} \rm
{\bf 1.} Consider the semigroup $S$ in Example \ref{Example Numerical Duplication}.2. It is GAS and, then, the previous theorem implies that also $T=S\! \Join^{47} \!(S-\langle K\rangle)$ is GAS. However, we notice that $2K\setminus K=\{44,63,91\}$, $3K \setminus 2K=\{85\}$ and $4K=3K$, while $2K(T) \setminus K(T)=\{135,173,217,229\}$ and $2K(T)=3K(T)$. \\
{\bf 2.} Despite Theorem \ref{Numerical duplication S-<K>}, if $\du$ is GAS for an ideal $I$ different form $S-\langle K \rangle$, it is not true that also $S$ is GAS. For instance, the semigroup $S$ in  Example \ref{Example Numerical Duplication}.1 is not GAS, but $S\! \Join^3 \! I$ is. 
\end{example}

\subsection{Dilatations of numerical semigroups}

We complete this section studying the transfer of the GAS property in a construction recently introduced in \cite{BS}: given $a \in M-2M$, the numerical semigroup $S+a=\{0\} \cup \{m+a \mid m \in M\}$ is called dilatation of $S$ with respect to $a$. 

\begin{proposition} \label{dilatation}
Let $a \in M-2M$. The semigroup $S+a$ is {\rm GAS} if and only if $S$ is {\rm GAS}.
\end{proposition}

\begin{proof}
We denote the semigroup $S+a$ by $T$. Recalling that $\F(T)=\F(S)+a$, by \cite[Lemma 3.1 and Lemma 3.4]{BS} follows that $2K(T)=2K(S)$ and 
\begin{equation*}
\begin{split}
&2K(S) \setminus K(S)=\{\F(S)-x_1, \dots, \F(S)-x_r, \F(S)\}, \\
&2K(T) \setminus K(T)=\{\F(T)-(x_1+a), \dots, \F(T)-(x_r+a), \F(T)\} 
\end{split}
\end{equation*}
for some $x_1, \dots, x_r \in M$. 

Assume that $S$ is a GAS semigroup. Then, $x_i$ is a minimal generator of $S$ and it is straightforward to see that $x_i+a$ is a minimal generator of $T$. Moreover, if $(x_i+a)-(x_j+a) \in \PF(T)$, then for every $m \in M$ we have $x_i-x_j+m+a \in T$, i.e. $x_i-x_j+m \in M$, that is a contradiction, since $S$ is GAS.

Now assume that $T$ is GAS. Suppose by contradiction that $x_i$ is not a minimal generator of $S$, that is $x_i=s_1+s_2$ for some $s_1$, $s_2 \in M$. We have $\F(S)-(s_1+s_2) \in 2K(S)\setminus K(S)$ and so $\F(S)-s_1 \in 2K(S)\setminus K(S)$, since $2K(S)$ is a relative ideal. Hence, $s_1=x_j$ for some $j$ and $(x_i+a)-(x_j+a)=s_2 \in S$. Since $x_i+a$ is a minimal generator, we have that $s_2 \notin T$. Moreover, for every $m+a \in M(T)$ we clearly have $s_2+m+a \in M(T)$, because $s_2 \in S$. This yields a contradiction because $(x_i+a)-(x_j+a)=s_2 \in \PF(T)$ and $T$ is GAS. Finally, if $x_i-x_j \in \PF(S)$, it is trivial to see that $x_i-x_j \in \PF(T)$.
\end{proof}

\begin{remark} \rm
Suppose $2K(S+a) \setminus K(S+a)=\{\F(S+a)-(x_1+a), \dots, \F(S+a)-(x_r+a), \F(S+a)\}$ with $x_1+a, \dots, x_r+a$ minimal generators of $S+a$, but $S+a$ is not GAS. Then $2K(S) \setminus K(S)=\{\F(S)-x_1, \dots, \F(S)-x_r, \F(S)\}$, but it is not necessarily true that $x_1, \dots, x_r$ are minimal generators of $S$. For instance, consider $S=\langle 7,9,11 \rangle$ and $S+7=\langle  14, 16, 18, 21, 23, 25, 27, 29, 38, 40 \rangle$. In this case $2K(S+7) \setminus K(S+7)=\{33-29,33-18,33\}$ and $2K(S) \setminus K(S)=\{26-22, 26-11, 26\}$.
\end{remark}

\end{document}